\documentclass[11pt]{article}

\usepackage[alphabetic]{amsrefs}
\usepackage{amsmath,amssymb,amsthm,enumerate,tikz,graphicx}
\usepackage[all,cmtip]{xy}
\usepackage[applemac]{inputenc}

\def \1{\mathds{1}}
\def \al{\alpha}

\def \bs{\backslash}
\def \C{{\mathbb C}}

\def \CB{{\cal B}}

\def \CE{{\cal E}}

\def \CP{{\cal P}}

\def \CT{{\cal T}}

\def \det{\operatorname{det}}
\def \df{\ \begin{array}{c} _{\rm def}\\ ^{\displaystyle =}\end{array}\ }

\def \ev{\mathrm{ev}}

\def \Ga{\Gamma}
\def \GL{\operatorname{GL}}
\def \ga{\gamma}

\def \Id{{\rm Id}}

\def \la{\lambda}

\def \N{{\mathbb N}}
\def \odd{\mathrm{odd}}
\def \OE{\mathrm{OE}}
\def \ol{\overline}

\def \SL{\operatorname{SL}}

\def \tr{\operatorname{tr}}

\def \val{\operatorname{val}}

\def \Z{{\mathbb Z}}
\def \({\left(}
\def \){\right)}
\def \[{\left[}
\def \]{\right]}

\newcommand{\e}
[1]{\emph{#1}\index{#1}}

\newcommand{\mat}
[4]{\(\begin{matrix}#1 & #2 \\ #3 & #4\end{matrix}\)}

\newcommand{\norm}
[1]{\left|\hspace{-1pt}\left|#1\right|\hspace{-1pt}\right|}

\renewcommand{\sp}
[1]{\left\langle #1\right\rangle}

\newcommand{\stack}
[2]{\genfrac{}{}{0pt}{1}{#1}{#2}}

\newcommand{\tto}
[1]{\stackrel{#1}{\longrightarrow}}

\newtheorem{theorem}{Theorem}[section]

\newtheorem{lemma}[theorem]{Lemma}

\newtheorem{proposition}[theorem]{Proposition}
\newtheorem{exmple}[theorem]{Example}
\newenvironment{example}[0]{\begin{exmple}\rm}
{\end{exmple}}
\newtheorem{defi}[theorem]{Definition}
\newenvironment{definition}[0]{\begin{defi}\rm}
{\end{defi}}
\newtheorem{remrk}[theorem]{Remark}

\newtheorem{exmples}[theorem]{Examples}
\newenvironment{examples}[0]{\begin{exmples}{\ }\\ 	\vspace{-20pt}\nopagebreak[4]
	\begin{itemize}\rm}{\end{itemize}\end{exmples}\vspace{5pt}}

\begin{document}

\pagestyle{myheadings} \markright{IHARA ZETA ON WEIGHTED GRAPHS}

\title{Ihara Zeta functions of infinite weighted graphs\\ \ \\ \small
SIAM J. discrete Math. Vol. 29, No. 4, 2100-2116 (2015)}
\author{Anton Deitmar\thanks{This research was funded by the DFG grant DE 436/10-1}}
\date{}
\maketitle

{\bf Abstract:}
The theory of Ihara zeta functions is extended to infinite graphs which are weighted and of finite total weight.
In this case one gets meromorphic instead of rational functions and the classical determinant formulas of Bass and Ihara hold true with Fredholm determinants.

$$ $$

\tableofcontents

\newpage
\section*{Introduction}

The Ihara zeta function, introduced by Yasutaka Ihara in the 1960s \cites{Ihara1,Ihara2} 
is a zeta function counting prime elements in discrete subgroups of rank one $p$-adic groups.
It can be interpreted as a geometric zeta function for the corresponding finite graph, which is a quotient of the Bruhat-Tits building attached to the $p$-adic group \cite{Serre}.
Over time it has been generalized in stages by Sunada, Hashimoto and Bass \cites{Sun1,Sun2,Hash0,Hash1,Hash2,Hash3,Hash4,Bass,Sunada}.
Comparisons with number theory can be found in the papers of Stark and Terras \cites{ST1,ST2,ST3}.
This zeta function is defined as the product
$$
Z(u)=\prod_p(1-u^{l(p)})^{-1},
$$
where $p$ runs through the set of prime cycles in a finite graph $X$. The product, being infinite in general, converges to a rational function, actually the inverse of a polynomial, and satisfies the famous \e{Ihara determinant formula}
$$
Z(u)^{-1}=\det(1-uA+u^2Q)(1-u^2)^{-\chi},
$$
where $A$ is the adjacency operator of the graph, $Q+1$ is the valency operator and $\chi$ is the Euler number of the graph.
One of the most remarkable features of the Ihara formula is, that in the case of $X=\Ga\bs Y$, where $Y$ is the Bruhat-Tits building of a $p$-adic group $G$ and $\Ga$ is a cocompact arithmetic subgroup of $G$ of split rank 1, then the right hand side of the Ihara formula equals the non-trivial part of the Hasse-Weil zeta function of the Shimura curve attached to $\Ga$, thus establishing the only known link between geometric and arithmetic zeta- or L-functions.

In recent years, several authors have asked for a generalization of these zeta functions to infinite graphs.
The paper \cite{Scheja} considers the arithmetic situation, where the graph is the union of a compact part and finitely many cusps. The zeta function is defined by plainly ignoring the cusps, so indeed, it is a zeta function of a finite graph.
In \cite{Clair1} and \cite{Clair2}, the zeta function of a finite graph is generalized to an $L^2$-zeta function where a finite trace on a group von-Neumann algebra is used to define a determinant.
In \cite{Grig}, an infinite graph is approximated by finite ones and the zeta function is defined as a suitable limit.
In \cites{Guido1,Guido2} a relative version of the zeta function is considered on an infinite graph which is acted upon by a group with finite quotient.
In \cite{Chinta}, finally, the idea of the Ihara zeta function is extended to infinite graphs by counting not all cycles, but only those which pass through a given point.

In this paper, infinite weighted graphs of finite total weight are considered.
It is shown that the Euler product of the zeta function converges to a meromorphic function without zeros, i.e., the reciprocal of an entire function.
Different known proofs of the Ihara formula in the finite case yield, when appropriately transferred to the infinite weighted case, indeed different results, here labelled as the Ihara-Sunada and the Bass-Ihara formula.
It is shown that twisting with local systems yields corresponding L-functions with similar properties.
In the case of tree lattices \cite{BassLub}, it may happen, that geodesics are reversed and so partial backtracking is allowed in the quotient graph.
We generalize this to arbitrary partial backtracking and find that the results persist in that case, too.

\section{Weighted graphs}\label{appB}

\begin{definition}
Let $X$ be a connected graph.
So $X$ consists of a vertex set $VX=V(X)$ and a set $EX=E(X)\subset \CP(VX)$, whose elements are subsets of $VX$ of order two, called \e{edges} of the graph.
Two vertices $x,y$ are called \e{adjacent} if $\{x,y\}$ is an edge.
For a vertex $x$ the \e{valency} $\val(x)$ is the number of edges having $x$ as an endpoint.
It can be infinite, and the graph is called \e{locally finite} if the valency is finite for every vertex.

\begin{center}
{\it In this paper we will only consider graphs of \e{bounded valency}, i.e. graphs $X$ for which there exists a constant $M>0$ such that $\val(x)\le M$ holds for ever vertex $x$.
}
\end{center}

A \e{path} in $X$ is a sequence of vertices $p=(x_0,x_1,\dots,x_n)$ such that for each $0\le j\le n-1$ the vertices $x_j$ and $x_{j+1}$ are adjacent.
The path is \e{closed} if $x_n=x_0$.
In that case we define the \e{shifted path} $\tau(p)$ as $(x_1,x_2,\dots,x_n,x_1)$.
The path $p$ is said  to be \e{reduced} or have \e{no backtracking}, if $x_{j-1}\ne x_{j+1}$ for every $1\le n\le n-1$.
A closed path $p=(x_0,\dots,x_n)$ is said to have a \e{tail}, if $x_1=x_{n-1}$.
A reduced path without tail is called a \e{regular}  path.
On the set of closed paths, we consider the equivalence relation generated by $p\sim\tau(p)$.
A \e{cycle} is an equivalence class of  paths and a \e{regular cycle} is a cycle consisting of regular paths only.
It follows that a path in a regular cycle has no tail.
A cycle is called \e{prime}, if it is not a power of a shorter one.
\end{definition}

\begin{definition}
Let $\OE=\OE(X)$ denote the set of all oriented edges of $X$.
So each edge in $EX$ give rise to two elements of $\OE$.
For an oriented edge $e$ we write $e^{-1}$ for its reverse.
\end{definition}

\begin{definition}
Let $w:\OE(X)\to (0,\infty)$ be a function, called the \e{weight function}.
The \e{total weight} of the graph $X$ is defined to be
$$
w(X)=\sum_{e\in \OE}w(e),
$$
and we will assume that it is finite, i.e.,
$$
w(X)<\infty.
$$
Note that if $OE(X)$ is uncountable (and the sum is interpreted as an integral with respect to the counting measure), then this finiteness condition implies that only countably many edges $e$ have non-zero weight.
As edges of weight zero do not contribute to what follows, we might as well remove them and assume that $OE(X)$ is countable.
If $e\in EX$ and $e_1,e_2\in \OE$ are the two possible orientations of $e$, then we write
$$
W(e)=w(e_1)w(e_2).
$$
For each path $p=(x_0,\dots,x_n)$ let
$$
w(p)=w(x_0,x_1)\cdots w(x_{n-1},x_n)
$$
be the weight of the path.
We then have $w(c^j)=w(c)^j$ for $j\in\N$ and the weight does not change under equivalence.
The length of a cycle $c$ is the number of edges and will be denoted by $l(c)$. We have $l(c^j)=jl(c)$.
\end{definition}

\begin{definition}
We define the \e{Ihara zeta function} of the weighted graph $X$ as the infinite product
$$
Z(u)=Z_X(u)=\prod_p\(1-w(p)u^{l(p)}\)^{-1},
$$
the product being extended over all regular prime cycles in $X$.
\end{definition}

\begin{definition}
Let $\ell^2(\OE)$ be the $\ell^2$-space on the set $\OE$, which we write as the set of all formal linear combinations $\sum_{e\in \OE}c_e e$ with
$
\sum_{e\in \OE}|c_e|^2<\infty.
$
On $\ell^2(\OE)$ define a linear operator $T$ by
$$
Te\df \sum_{e'} w(e')e',
$$
where the sum runs over all oriented edges $e'$ such that the origin vertex $o(e')$ of $e'$ equals
the target vertex $t(e)$ of $e$, and $e'\ne e^{-1}$.
\end{definition}

Recall that for a trace class operator $T$ on a Hilbert space the \e{Fredholm determinant} is defined by
$$
\det(1-T)=\sum_{k=0}^\infty (-1)^k\tr\bigwedge^k(T).
$$
Then $\det(1-S)(1-T)=\det(1-S)\det(1-T)$ and the function
$$
u\mapsto \det(1-uT)
$$
is an entire function with zeros at $1/\la$, where $\la$ is an eigenvalue of $T$.
For small values of $u$, the value $\det(1-uT)$ is close to one and it satisfies 
$$
\log\det(1-uT)=\tr\log(1-uT)=-\sum_{n=1}^\infty\frac{u^n\tr T^n}n.
$$
For this see \cite{Simon}.

\begin{theorem}\label{thm1.1}
The operator $T$ is of trace class.
The infinite product $Z(u)$ converges for $|u|$ sufficiently small. The limit function extends to a meromorphic function on $\C$ without zeros, more precisely, the function $Z(u)^{-1}$ is entire and satisfies
$$
Z(u)^{-1}=\det(1-uT).
$$
\end{theorem}

In the case of a finite graph with weight one this goes back to an idea of Hashimoto \cite{Hash1}, which later was refined by Bass \cite{Bass}.

\begin{proof}
We show that the operator $T$ is of trace class, and for every $n\in\N$ we have
$$
\tr T^n= \sum_{l(c)=n} l(c_0)\,w(c),
$$
where the sum runs over all  regular cycles $c$ of length $n$ and $c_0$ is the underlying prime to $c$.

For this we consider the natural orthonormal basis of $\ell^2(\OE)$  given by $(e)_{e\in \OE}$.
Using this orthonormal basis, one sees that the trace of $T^n$ is as claimed, once we know that $T$ is of trace class.
Let $M$ be an upper bound for the valency of the graph $X$, then one has
$$
\sum_e\norm{Te}\le M\sum_ew(e)<\infty
$$
and hence $T$ is of trace class.
For small values of $u$ we have
\begin{align*}
\det(1-uT) &=\exp\left( -\sum_{n=1}^\infty\frac{u^n}n\tr T^n\right)\\
&=\exp\left( -\sum_{n=1}^\infty\frac{u^n}n\sum_{l(c)=n}l(c_0)\,w(c)\right)\\
&=\exp\left( -\sum_{c_0}\sum_{m=1}^\infty \frac{u^{ml(c_0)}}m \,w(c_0)^m\right)\\
&= \prod_{c_0}\left( 1-w(c_0)u^{l(c_0)}\right)= Z(u)^{-1}.
\end{align*} 
This a fortiori also proves the convergence of the product.
\end{proof}

\begin{examples}
\item Interesting examples arise from naturally arising weight functions.
The first is an infinite quotient graph by a tree lattice $\Ga$ of finite covolume as in \cite{BassLub}.
Here a natural weight for an edge $e$ is given by $1/|\Ga_e|$, where $\Ga_e$ is the stabilizer group of $e$.
There will be more on this example in Section \ref{partbacktrack}.
\item Another class of natural examples is Cayley graphs: let $G$ be a group finitely generated by a subset $S$ and let $X$ be its Cayley-graph.
The nodes of $X$ are the elements of $G$ and for any $s\in S$ the node $g$ is connected to the node $gs$.
If we assume that $1\notin S$ and $S=S^{-1}$, then the oriented edges of $X$ are in natural bijection with $G\times S$.
Let
$l(g)$ be the word-length for $g\in G$.
If the sum $\sum_{g\in G}\frac1{(l(g)+1)^k}$  converges for some $k\in\N$, we can choose the weight $w(g,s)=\frac1{(l(g)+1)^k}$. Already in the simplest cases like $G=\Z^d$, the spectrum of the operator $T$ turns out to be quite hard to calculate and we reserve this for subsequent papers.
In the unweighted case the spectrum has in certain examples been calculated in \cite{Angel}.  
\end{examples}

\section{The Ihara-Sunada formula}\label{sec2}
For $m\in\N$ let
$$
N_m=\sum_{\substack{ l(p)=m\\
\text{regular, closed}\\
\text{no tail}}}w(p)
$$
where the sum runs over all regular closed paths without tail.

\begin{lemma}\label{lem2.1}
For the zeta function $Z(u)$ we have
$$
u\frac{Z'}Z(u)=\sum_{m=1}^\infty N_mu^m.
$$
\end{lemma}

\begin{proof}
Taking the logarithm, we get
$$
\log Z(u)=\sum_{c_0}\sum_{n=1}^\infty \frac{w(c_0)^n}nu^{nl(c_0)}
=\sum_c\frac{w(c)u^{l(c)}}{l(c)}l(c_0),
$$
where in the first identity, $c_0$ runs through all regular prime   cycles and in the second, $c$ runs through all   regular cycles where $c_0$ denotes the prime underlying $c$.
We take the derivative and multiply by $u$ to get
$$
u\frac{Z'}Z(u)=\sum_cw(c)u^{l(c)}l(c_0).
$$
Replacing the sum over all regular cycles with a sum over all regular closed paths without tails, the factor $l(c_0)$ drops and we get the claim of the lemma.
\end{proof}

We form the $\ell^2$-space $\ell^2(VX)$ of vertices on which we consider the \e{adjacency operator} $A_1$ defined by
$$
A_1(x) =\sum_{x'}w(x,x')x',
$$
where the sum runs over all vertices $x'$ adjacent to $x$.
As the weight of $X$ is finite, the adjacency operator is a bounded operator.
For $m\in\N$ we set
$$
A_m(x)=\sum_{l(p)=m}w(p)x_p,
$$
where the sum extends over all regular paths $p$ of length $m$, starting in $x$, where $x_p$ denotes the endpoint of the path $p$.
 We finally put $A_0=\Id$, the identity operator.

Further set $B_0=\Id$ and let
\begin{align*}
B_{2n}x&=\sum_{x'} W(x,x')^{n}x,&n\ge 1,\\
B_{2n+1}x&=\sum_{x'}W(x,x')^{n}w(x,x')x',&n\ge 0,
\end{align*}
where the sums run over all neighbors $x'$ of $x$.
Note that $B_1=A_1$.

\begin{lemma}\label{lem2.2}
For every $m\in\N$ we have
$$
\sum_{j=0}^m(-1)^jA_{m-j}B_j=0.
$$
This is equivalent to the identity of operator-valued formal power series:
$$
A(u)B(-u)=1,
$$
where $A(u)=\sum_{j=0}^\infty u^jA_j$ and $B(u)=\sum_{j=0}^\infty u^jB_j$.
\end{lemma}

\begin{proof}
One best thinks of the operator $A$ as sending potentials from a vertex $v$ to all its neighbors.
Accordingly, $A_m$ sends potentials along all reduced paths.
Finally, $B_{2n}$ may be viewed as sending potentials to all neighbors and then sending them back, where this process is repeated $n$-times.
The operator $A_{m-1}B_1$ sends potentials first to all neighbors, then in all directions along reduced paths.
Hence
$$
A_{m-1}B_1x=\sum_pw(p)x_p,
$$
where the sum runs over all paths $p$, starting at $x$, which can have backtracking at the first step, but not later.
Hence $A_{m-1}B_1-A_m$ sees only paths which have backtracking at the first step. So this operator coincides with $A_{m-2}B_2$ except that the latter also sees paths which also have  backtracking at the second step.
These, again, are taken care of by $A_{m-3}B_3$ up to  backtracking at the third step and so on. We end up with
$$
A_{m-1}B_1-A_m=A_{m-2}B_2-A_{m-3}B_3+\dots\pm B_m
$$
which is equivalent to $A(u)B(-u)=1$.
\end{proof}

Let $Q$ be the operator defined by $Qx=q(x)x$, where the number of neighbors of the vertex $x$ is $q(x)+1$.
In other words, $Q+1$ is the valency operator.

\begin{lemma}\label{lem2.3}
Let $B_\ev(u)=\sum_{\stack{j\ge 0}{\text{even}}}u^jB_j$ and $B_\odd(u)=B(u)-B_\ev(u)$.
For any vertex $x$ we have
$$
B_\ev(u)x=x+\sum_{x'}\(\frac{u^2W(x,x')}{1-u^2W(x,x')}\)x
$$
and
$$
B_\odd(u)x=\sum_{x'}\(\frac{u\,w(x,x')}{1-u^2W(x,x')}\)x'
$$
\end{lemma}

\begin{proof}
A calculation.
\end{proof}

\begin{example}
In the special case of a finite graph with constant weight $w=1$ we get
$$
B(u)=\frac1{1-u^2}(1+uA_1+u^2Q),
$$
and so
$$
A(u)=(1-u^2)(1+uA_1+u^2Q)^{-1}.
$$
\end{example}

Back to the general case, 
for $m,n\ge 1$ define the operator $C_{m,n}$ by 
$$
C_{m,n}x=\sum_{\substack{l(p)=m\\ \text{regular}\\
x_0=x}}W(x_0,x_1)^{n}w(p)x_m,
$$
where the sum runs over all paths without of length $m$, starting at $x$. Further, $x_0$ and $x_1$ are the first two vertices of the path $p$ and $x_m$ is the last.

\begin{lemma}\label{lem2.4}
$N_0=N_1=N_2=0$, and for $m\ge 3$ we have
$$
N_m=\tr A_m-\sum_{j=1}^{\lfloor\frac {m-1}2\rfloor}\tr A_{m-2j}B_{2j}+2\sum_{j=1}^{\lfloor\frac {m-1}2\rfloor}\tr C_{m-2j,j}.
$$
\end{lemma}

\begin{proof}
We compute
\begin{align*}
N_m&= \sum_{\substack{l(p)=m\\
\text{reg. closed}\\
\text{no tail}}}w(p)=\tr A_m-\sum_{\substack{l(p)=m\\
\text{reg. closed}\\
\text{with tail}}}w(p)\\
&=\tr A_m-\tr A_{m-2}B_2+\overbrace{\sum_{x,x'}W(x,x')
\sum_{\substack{
l(p)=m-2\\
\text{reg. closed}\\
x_0=x'\\
x_1=x\ \text{or}\ x_{m-3}=x
}}w(p)}^{=*}
\end{align*}
The first sum in the last line ranges over all pairs $x,x'$ of neighbored vertices.
The last sum contains those contributions, for which only one of the last two conditions is satisfied.
By reversing the path, one condition becomes the other, so that we get
\begin{align*}
*&=2\sum_{x,x'}W(x,x')
\sum_{\substack{
l(p)=m-2\\
\text{reg. closed}\\
x_0=x',\ x_1=x,\\ \text{no tail}
}}w(p)\\
&\quad+\sum_{x,x'}W(x,x')
\sum_{\substack{
l(p)=m-2\\
\text{reg. closed}\\
x_0=x',\ x_1=x,\\ \text{with tail} 
}}w(p)\\
&= 2\tr C_{m-2,1}-
\sum_{x,x'}W(x,x')
\sum_{\substack{
l(p)=m-2\\
\text{reg. closed}\\
x_0=x', x_1=x,\\
\text{with tail}
}}w(p)
\end{align*}
In the last sum, the tail can be collapsed again, giving another factor of $W(x,x')$, so that we can write it as $\tr A_{m-4}B_4$ minus a contribution of the form $2\tr C_{m-4,2}$ plus another sum over paths with tails and so on.
We end up with the claim.
\end{proof}

\begin{proposition}
Let $C(u)=\sum_{m,n=1}^\infty u^{m+2n}C_{m,n}$. Then
$$
u\frac{Z'}Z(u)=\tr\left[ (A(u)-1)(2-B_\ev(u))+2C(u)\right].
$$
\end{proposition}

\begin{proof}
This is just a reformulation of Lemma \ref{lem2.4}.
\end{proof}

\begin{lemma}\label{lem2.6}
For all $m,n\in\N$ one has
$$
\tr C_{m,n}=\sum_{j=0}^{2n}(-1)^j\tr A_{m+j}B_{2n-j}.
$$
\end{lemma}

\begin{proof}
We first show that for $m\ge 3$ we have
\begin{align}\label{equ1}
\tr C_{m,n}=\tr A_{m-1}B_{2n+1}-\tr A_{m-2}B_{2n+2}+\tr C_{m-2,n+1}.
\end{align}
For this we compute
\begin{align*}
\tr C_{m,n}&=\sum_{x,x'}W(x,x')^{n}\sum_{\substack{l(p)=m\\ \text{regular, closed}\\
x_0=x,\ x_1=x'}}w(p)\\
&=\tr A_{m-1}B_{2n+1}-\sum_{x,x'}W(x,x')^{n}\sum_{\substack{l(p)=m-1\\ \text{regular}\\
x_0=x',\ x_1=x_{m-1}=x}}w(p)\\
&=\tr A_{m-1}B_{2n+1}-\tr A_{m-2}B_{2n+2}+\tr C_{m-2,n+1}.
\end{align*}
Replacing $m$ with $m+2$ and $n$ with $n-1$ the equation (\ref{equ1}) becomes
\begin{align}\label{equ2}
\tr C_{m,n}=\tr A_mB_{2n}-\tr A_{m+1}B_{2n-1}+\tr C_{m+2,n-1},
\end{align}
which now holds for all $m\ge 1$.
Applying equation (\ref{equ2}) to the last summand of itself and repeating this step, one gets
\begin{align*}
\tr C_{m,n}&=tr A_mB_{2n}-\tr A_{m+1}B_{2n-1}+\tr C_{m+2,n-1}\\
&=tr A_mB_{2n}-\tr A_{m+1}B_{2n-1}\\
&\ +\tr A_{m+2}B_{2n-2}-\tr A_{m+3}B_{2n-3}+tr C_{m+4,n-2}\\
&=\sum_{\nu=0}^{2n}(-1)^j\tr A_{m+j}B_{2n-j}.
\tag*\qedhere
\end{align*}
\end{proof}

We write
$$
A'(u)=\sum_{m=1}^\infty mu^{m-1}A_m
$$
for the formal derivative of $A(u)$.
For $m\in\N$ we also write
$$
A_{\ge m}(u)=\sum_{j\ge m}u^jA_j.
$$

\begin{lemma}
We have
$$
\tr C(u)=\tr\left[ \frac u2A'(u)B(-u)+\frac12A_\odd(u)B(u)-A(u)+1\right].
$$
\end{lemma}

\begin{proof}
Using Lemma \ref{lem2.6} we compute
\begin{align*}
&\tr C(u)\\
&=
\sum_{m,n=1}^\infty u^{m+2n}\tr C_{m,n}\\
&=\sum_{m,n=1}^\infty u^{m+2n}\sum_{j=0}^{2n}(-1)^j\tr A_{m+j}B_{2n-j}\\
&=\sum_{m=1}^\infty u^m\tr\left[\(\sum_{i=0}^\infty u^{i}A_{m+i}\)\(\sum_{j=0}^\infty(-u)^jB_{j}\)\right]^{\rm even}-u^m\tr A_m\\
&=\sum_{m=1}^\infty u^m\tr\left[
u^{-m}A_{\ge m}(u)
B(-u)
\right]^{\rm even}-u^m\tr A_m\\
&=\frac12\sum_{m\ge 1}u^m\tr\(u^{-m}A_{\ge m}(u)B(-u)+(-1)^mu^{-m}A_{\ge m}(-u)B(u)\)\\
&\ -\tr A_{\ge 1}(u)\\
&=\tr\left[\frac12\sum_{m\ge 1}A_{\ge m}(u)B(-u)
+\frac12\sum_{m\ge 1}(-1)^mA_{\ge m}(-u)B(u)
- A_{\ge 1}(u)\right]
\end{align*}
Now note that 
\begin{align*}
\sum_{m=1}^\infty A_{\ge m}(u)=\sum_{m=1}^\infty mu^mA_m
=uA'(u)
\end{align*}
and
$$
\sum_{m\ge 1} (-1)^mA_{\ge m}(-u)=A_\odd(u).
$$
The claim follows.
\end{proof}

\begin{lemma}\label{lem2.9}
We have
$$
\frac{Z'}Z(u)=\tr\left[ \frac1u\left[B_\ev(u)-1\right]+A'(u)B(-u)\right].
$$
\end{lemma}

\begin{proof}
We compute
\begin{align*}
-\frac{Z'}Z(u)&=
\tr\left[\frac1uA_{\ge 1}(u)(B_\ev(u)-2)-\frac 2uC(u)\right]\\
&= \tr\[ \frac1uA_{\ge 1}(u)(B_\ev(u)-2)\right.\\
&\ \ \ \ \left.
- A'(u)B(-u)-\frac1uA_\odd(u)B(u)+\frac 2uA_{\ge 1}(u)\]\\
&= \tr\[ \frac1uA_{\ge 1,\ev}(u)B_\ev(u)+\frac1uA_\odd(u)B_\ev(u)\right.\\
&\ \ \ \ \left.-\frac1uA_\odd(u)B_\ev(u)-\frac1uA_\odd(u)B_\odd(u)-A'(u)B(-u)\]\\
&= \tr\[\frac1u\[A_{\ge 1}(u)B(-u)\]^{\rm even}-A'(u)B(-u)\].
\end{align*}
By $A_{\ge 1}(u)=A(u)-1$ and $A(u)B(-u)=1$, this implies the claim.
\end{proof}

The Fredholm determinant can be extended to be applicable  to a formal power series of the form
$$
1+T(u)=1+T_1 u+T_2 u^2+\dots,
$$
where each $T_j$ is a trace class operator on some Hilbert space $H$ by defining
$$
\det(1+T(u))=\sum_{k=0}^\infty (-1)^k\tr\wedge^k T(u),
$$
where $\wedge^k T(u)$ is considered an element of $\CT\(\wedge^k H\)[[u]]$ and $\CT$ being the algebra of trace class operators.

For an oriented edge $e\in\OE(X)$  write $o(e)$ for its starting vertex and $\tau(e)$ for its terminal vertex.
Write $C_0=\ell^2(V X)$ and $C_1=\ell^2(\OE X)$ and let $\sigma:C_0\to C_1$ be defined by
$$
\sigma(x)=\sum_{e:o(e)=x}w(e)e.
$$
Further let $J:C_1\to C_1$ be the weighted flip, i.e., 
$$
J(e)=w(e^{-1})e^{-1},
$$
where $e^{-1}$ is the reverse of the oriented edge $e$.
For $n\ge 1$ one has
$$
B_n=\tau J^{n-1}\sigma
$$
and so
$$
B(u)=1+u\tau(1-uJ)^{-1}\sigma.
$$

\begin{theorem}
[Ihara-Sunada formula]\label{thm2.10}
For weighted graphs, which may be infinite, but have finite total weight, the entire function $Z(u)^{-1}$ can be written as
$$
Z(u)^{-1}=\det\(1-u\tau(1+uJ)^{-1}\sigma\)\,\prod_{e\in EX} \(1-u^2W(e)\),
$$
the product over all edges $e$ converges to an entire function with zeros at $\pm 1/\sqrt{W(e)}$, where $e$ runs through the set of edges.
\end{theorem}

\begin{proof}
Using the orthogonal basis of $VX$ given by the vertices, we compute
\begin{align*}
\tr\left[\frac1u(1-B_\ev(u))\right]
&= -\frac1u\sum_x\sum_{x'\sim x}\(\frac1{1-u^2W(x,x')}-1\)\\
&= -\frac 2u\sum_e\frac{u^2W(e)}{1-u^2W(e)}\\
&= \sum_e\(\log(1-u^2W(e))\)'
\end{align*}
We integrate this function from zero to $u$ and take the exponential to get the product
$$
\prod_e \(1-u^2W(e)\).
$$
Since $\sum_{e}W(e)<\infty$, this product converges everywhere to an entire function.
Note that $B(-u)=A(u)^{-1}$, so the first factor can also be written as $\det(A(u))^{-1}=\det(B(-u))$.
\end{proof}

{\bf Remark.}
In the special case of a finite graph and weight one these factors are
$$
\det(B(-u))=\det\(\frac{1-uA+u^2Q}{1-u^2}\)
$$
and
$$
\prod_e (1-u^2W(e))=
(1-u^2)^{|EX|},
$$
where $EX$ is the set of edges.
So that in total one gets the classical Ihara formula
$$
Z(u)^{-1}=(1-u^2)^{-\chi}\det(1-uA+u^2Q),
$$
where $\chi=|VX|-|EX|$ is the Euler number of the graph, which is $\le 0$ if $X$ is not a tree.

\section{The Bass-Ihara formula}
In this section we follow the approach of Bass in \cite{Bass} to the Ihara formula. It turns out that in the case of weighted graphs some natural deformation cannot occur and so Bass's approach leads to a different kind of determinant formula.
The title of this section refers to the fact that Bass's approach can be viewed as a deformation of the homology complex.
Let $C_0=\ell^2(VX)$ and $C_1=\ell^2(\OE(X))$ and consider the  weighted flip $J(e)=w(e^{-1})e^{-1}$ which maps an oriented edge to its inverse times its weight.
Consider also
the operator
$\sigma: C_0\to C_1$ defined by
\begin{align*}
\sigma(x)&=\sum_{e:o(e)=x}w(e)e.
\end{align*}

\begin{theorem}[Bass-Ihara formula]
One has
$$
Z(u)^{-1}=\det\left[ 1+u\mat {uB_2-A}{u\tau J}\sigma J\right].
$$
\end{theorem}

\begin{proof}
We consider the following operators on $C_0\oplus C_1$:
\begin{align*}
L&=\mat 1 {u\tau J-\tau} 0 {1-uJ} &M&=\mat 1{\tau} {u\sigma}{1+u J},
\end{align*}
where one stands for the identity operator on the respective spaces.
A computation shows
$$
ML=\mat 1 0 {u\sigma}{(1-uT)(1-uJ)}
$$
and
\begin{align*}
LM&= 1+u\mat{-A} 0 {\sigma}0+u^2\mat {B_2}{\tau J^2}{J\sigma}{J^2}\\
&= \mat 1\ \ {1-uJ} \left[ 1+u\mat {uB_2-A}{u\tau J}\sigma J\right].
\end{align*}
As $\det(LM)=\det(ML)$ we conclude the claim.
\end{proof}

\section{Twisting with local systems}
A local system is the same as a locally constant sheaf.
To give a local system of complex vector spaces on $X$ is the same as giving a finite dimensional representation $\tau:\Ga\to\GL(V)$ of the fundamental group $\Ga=\pi_1(X)$, which for this purpose may be defined as the group of deck transformations on the universal covering $\tilde X$. 
In that case the set $\Ga\bs\(\tilde X\times V\)$ is the total space of the sheaf and sections can be described as maps $s:\tilde X\to V$ satisfying $s(\ga x)=\tau(\ga)s(x)$ for all $x\in\tilde X$, $\ga\in\Ga$.
We will only consider \e{hermitian} local systems, which means that we assume a given inner product on $V$ such that the representation $\tau$ is unitary.

Fixing a base point $x_0\in X$, the group $\Ga$ can be identified with the fundamental group $\pi_1(X,x_0)$ at the point $x_0$ and every closed path $p$ in $X$ defines a conjugacy class $[\ga_p]$ in $\Ga$.
The twisting results in a replacement of the zeta function by a so called L-function which is defined to be
$$
L(\tau,u)\df \prod_{[p]}\det\left(1-w(p)u^{l(p)}\tau(\ga_p)\right)^{-1},
$$
The results of the previous section generalize to this situation.
One only has to make clear how to interpret the statements.
One way to view a hermitian local system on a graph $X$ is to say that one attaches a hermitian vector space $V_x$ to every vertex $x$ and for each pair $x,x'$ of adjacent vertices one has a unitary operator $T_{x,x'}:V_x\tto\cong V_{x'}$, called the \e{transfer operator} or \e{parallel transport}.
One then defines operators $A$ and $B$  on the Hilbert space of square integrable sections $s$ of the system given by
$$
As(x)=\sum_{x'}w(x,x')T_{x',x}s(x'),
$$
where the sum runs over all vertices $x'$ adjacent to $x$.
For any path $p=(x_0,\dots,x_n)$ one gets a unitary map $T_p:V_{x_0}\to V_{x_n}$ by composing the local transfer operators. Denote by $p^{-1}$ the path in the opposite direction, so $p^{-1}=(x_n,x_{n-1},\dots,x_0)$.
One sets
$$
A_ms(x)=\sum_pw(p)T_{p^{-1}}s(x_p).
$$
Let $B_0=\Id$ and let
\begin{align*}
B_{2n}s(x)&=\sum_{x'} W(x,x')^{n}s(x),&n\ge 1,\\
B_{2n+1}s(x)&=\sum_{x'}W(x,x')^{n}w(x,x')T_{x',x}s(x'),&n\ge 0,
\end{align*}
where the sums run over all neighbors $x'$ of $x$.
Finally let $B(u)=\sum_{j=0}^\infty u^jB_j$

The hermitian local system also gives a hermitian vector space  $V_e$ for each edge $e$ isomorphism $\tilde\tau_e:V_e\tto\cong V_{\tau(e)}$ and $\tilde o_e:V_{e}\tto\cong V_{o(e)}$, as well as $V_e\tto\cong V_{e^{-1}}$, the latter being compatible in the obvious way with the former.
Let $C_0$ denote the $\ell^2$-sum of all $V_x$ for $x\in VX$ and let $C_1$ denote the $\ell^2$-sum of all $V_e$ with $e\in\OE X$.  Let $\sigma:C_0\to C_1$ be defined by
$$
\sigma(v_x)=\sum_{e:o(e)=x}w(e)\tilde o^{-1}(v_x),\qquad v_x\in V_x.
$$
Further let $J:C_1\to C_1$ be the weighted flip, i.e., 
$$
J(v_e)=w(e^{-1})v_{e^{-1}},\qquad v_e\in V_e,
$$
where $e^{-1}$ is the reverse of the oriented edge $e$ and $v_{e^{-1}}$ is the vector naturally identified with $v_e$.
For $n\ge 1$ one has
$$
B_n=\tau J^{n-1}\sigma
$$
and so
$$
B(u)=1+u\tau(1-uJ)^{-1}\sigma.
$$

\begin{theorem}
The product $L(\tau,u)$ converges for $|u|$ small enough to a meromorphic function such that the reciprocal $L(\tau,u)^{-1}$ is entire.
One has
$$
Z(u)^{-1}=\det\(1-u\tau(1+uJ)^{-1}\sigma\)\,\prod_{e\in EX} \(1-u^2W(e)\)^{\dim V},
$$
\end{theorem}

\begin{proof}
The proof of Theorem \ref{thm2.10} can be applied.
\end{proof}

\section{Partial backtracking}\label{partbacktrack}

A typical case of a weighted graph as above arises in the theory of reductive groups over local fields and their Bruhat-Tits buildings, see \cites{Serre, Lub1,Lub2}.
So let $k$ be a nonarchimedean local field and consider the group $G=\SL_2(k)$.
If $\mathrm{char}(k)>0$ then there can be a lattice $\Ga\subset G$ which is not uniform.
This means that $\Ga$ is a discrete subgroup of finite covolume $\mathrm{vol}(\Ga\bs G)<\infty$.
The Bruhat-Tits building $\CB$ is a tree which is acted upon by $G$. The stabilizer of an edge is a compact open subgroup $H$ and the vertices decompose into finitely many $G$-orbits.
We are interested in the quotient graph $X=\Ga\bs\CB$.
The fact that $\Ga$ has finite covolume results in
$$
\sum_{e\in VE}\frac 1{|\Ga_e|}<\infty,
$$
where $|\Ga_e|$ is the finite number of elements of the $\Ga$-stabilizer of any preimage $\tilde e$ in $\CB$ of $e$.
So in this case there is a natural weight function $w(e)=1/|\Ga_e|$.
However, there's more.

In the case when $X$ is finite and $\Ga$ its fundamental group, then each regular cycle defines a unique free homotopy class in $[S^1:X]$. The latter set can be identified with the conjugacy classes $[\ga]$ in the group $\Ga$.
This results in a bijection between the set of regular cycles and the set of non-trivial conjugacy classes in $\Ga$ which can be understood as follows.
Let $c$ be a regular cycle. Its preimage in the universal covering $\tilde X\to X$ consists of a union of infinite regular paths, which are all conjugate to each other under $\Ga$.
Pick one of those $p$, then $p$ maps surjectively onto $c$, so, as $c$ is finite, there exists an element $\ga\in\Ga$ which maps $p$ to itself and $c\cong p/\sp\ga$.
The bijection maps $c$ to the conjugacy class $[\ga]$ of $\ga$.

Now back to the case of the quotient of a Bruhat-Tits building by a lattice $\Ga$.
In this case, we consider the projection $\CB\to X=\Ga\bs\CB$. It happens that an infinite regular path $p$ in $\CB$ is mapped to itself by a member $\ga\in\Ga$ which at the same time fixes a point $x_0$ on $p$.
Then $x_0$ is either a vertex or the mid-point of an edge. In the latter case, we get a loop on $X$ which can be removed without changing the zeta function by introducing new vertices of weight one. 
More interesting things happen when $x_0$ is a vertex. Then $p$ is mapped to a path with backtracking, which is to say, that in order to capture all quotients of infinite paths  in $\CB$, we have to allow partial backtracking.
This means that we have to sign out a set $\CE$ of oriented edges at which to allow backtracking.
A path $p=(x_0,\dots,x_n)$ in $X$ is called \e{$\CE$-regular}, if, whenever $x_{j-1}=x_{j+1}$, then the edge $(x_{j-1},x_j)$ belongs to $\CE$. 
By an \e{$\CE$-cycle} we mean an equivalence class of closed paths consisting of $\CE$-regular paths only.
Then we form the zeta function
$$
Z_{\CE}(u)\df \prod_{p}\left(1-w(p)u^{l(p)}\right)^{-1},
$$
where the product runs over all prime $\CE$-cycles.
We first form the operator $T_\CE$ on the Hilbert space $\ell^2(\OE)$, given by 
$$
T_\CE e\df \sum_{e'} w(e')e',
$$
where this time the sum runs over all oriented edges $e'$ such that $o(e')=t(e)$, and $e'\ne e^{-1}$ unless $e\in\CE$.

\begin{theorem}\label{thm5.1}
The operator $T_{\CE}$ is of trace class.
The infinite product $Z_{\CE}(u)$ converges for $|u|$ sufficiently small. The limit function extends to a meromorphic function on $\C$ without zeros, more precisely, the function $Z_{\CE}(u)$ is entire and satisfies
$$
Z_{\CE}(u)^{-1}=\det(1-uT_{\CE}).
$$
\end{theorem}

\begin{proof}
Same as the proof of Theorem \ref{thm1.1}.
\end{proof}

For $m\in\N$ we set
$$
A_{\CE,m}(x)=\sum_{l(p)=m}w(p)x_p,
$$
where the sum extends over all $\CE$-regular paths $p$ of length $m$, starting in $x$, where $x_p$ denotes the endpoint of the path $p$.
Finally we define $A_{\CE,0}$ to be the Identity operator.

Set $B_{\CE,0}=\Id$ and let
\begin{align*}
B_{\CE,1}x&=\sum_{x'}w(x,x')x'\\
B_{\CE,2}x&=\sum_{\substack{x'\\ x\to x'\notin\CE}} W(x,x')x\\
B_{\CE,2n+1}x&=\sum_{\substack{x'\\ x\to x'\notin\CE\\ x'\to x\notin\CE}}W(x,x')^{n}w(x,x')x',&n\ge 1,\\
B_{\CE,2n}x&=\sum_{\substack{x'\\ x\to x'\notin\CE\\ x'\to x\notin\CE}} W(x,x')^{n}x,&n\ge 2,
\end{align*}
where the sums run over all neighbors $x'$ of $x$, leaving out those for which $(x,x')$ or $(x',x)$ belongs to $\CE$ as indicated.
Finally, let $q_\CE(x)+1$ denote the number of neighbors $x'$ of $x$ with $(x,x')\notin\CE$ and let $Q_\CE x=q_\CE(x)x$.

We also define $J_\CE$ on $C_1=\ell^2(\OE X)$ by
$$
J_\CE(e)=\begin{cases} w(e^{-1})e^{-1}& e\notin\CE\text{ and }e^{-1}\notin\CE,\\
0&\text{otherwise.}\end{cases}
$$
Then we have
$$
B_{\CE,m}=\tau J_\CE^{m-1}\sigma
$$
for $m\ge 3$, so that
$$
B_\CE(u)=1+uB_{\CE,1}+u^2B_{\CE,s}+u^3\tau J^2(1-uJ)^{-1}\sigma.
$$

\begin{lemma}
We have
$$
A_\CE(u)B_\CE(-u)=1,
$$
where $A_\CE(u)=\sum_{j=0}^\infty u^jA_{\CE,j}$ and $B_\CE(u)=\sum_{j=0}^\infty u^jB_{\CE,j}$.
\end{lemma}

\begin{proof}
The proof is a straightforward generalization of the proof of Lemma \ref{lem2.2}.
\end{proof}

Let $B_{\CE,\ev}(u)=\sum_{\stack{j\ge 0}{\text{even}}}u^jB_{\CE,j}$ and $B_{\CE,\odd}(u)=B_\CE(u)-B_{\CE,\ev}(u)$.
A calculation shows that
\begin{align*}
B_{\CE,\ev}(u)x&=x+\sum_{\substack{x\to x'\notin\CE\\ x'\to x\notin\CE}}\(\frac{u^2W(x,x')}{1-u^2 W(x,x')}\)x\\
&\ \ \ -u^2\sum_{\substack{x\to x'\notin\CE\\ x'\to x\in\CE}}W(x,x')x\\
B_{\CE,\odd}(u)&=\sum_{\substack{x\to x'\notin\CE\\ x'\to x\notin\CE}}\(\frac{uw(x,x')}{1-u^2W(x,x')}\) x'
-u\sum_{\substack{x\to x'\in\CE\\ \text{or}\\ x'\to x\in\CE}}w(x,x')x'.
\end{align*}

As in the beginning of Section \ref{sec2}, for $m\in\N$ we define
$$
N_{\CE,m}=\sum_{\substack{l(p)=m\\ \CE\text{-regular, closed}\\ \text{at most }\CE\text{-regular tail}}}w(p).
$$
Here we say that a closed path $p=(x_0,\dots,x_n)$ has \e{at most an $\CE$-regular tail}, if
$$
x_0=x_{n-1}\quad\Rightarrow\quad (x_1,x_0)\in\CE.
$$
Along the lines of Lemma \ref{lem2.1} we get
$$
u\frac{Z_{\CE}'}{Z_{\CE}}(u)=\sum_{m=1}^\infty N_{\CE,m}u^m.
$$

\begin{lemma}
With
$$
C_{\CE,m,n}x=\sum_{\substack{l(p)=m\\ \CE\text{-regular}\\ x_0=x}}W(x_0,x_1)^{n}w(p)x_m
$$
we have $n_{\CE,0}=N_{\CE,1}=0$ and for $m\ge 2$,
$$
N_{\CE,m}=\tr A_{\CE,m}-\sum_{j=1}^{\left[\frac{m-1}2\right]}\tr A_{\CE,m-2j}B_{\CE,2j}+2\sum_{j=1}^{\left[\frac{m-1}2\right]}\tr C_{\CE,m-2j,j},
$$
which with $C_\CE(u)=\sum_{m,n=1}^\infty u^{m+2n}C_{m,n}$ can be written as
$$
u\frac{Z_\CE'}{Z_\CE}(u)=\tr\left[ A_\CE(u)-1)(2-B_{\CE,\ev}(u))+2C_\CE(u)\right].
$$
We end up with
$$
\frac{Z_\CE'}{Z_\CE}(u)=\tr\left[\frac1u[B_{\CE,\ev}(u)-1]+A_\CE'(u)(B_\CE(-u)\right].
$$
\end{lemma}

The proof is similar to Lemmas \ref{lem2.4} to \ref{lem2.9}.

\begin{definition}
The set $\CE$ of exceptional edges contains oriented edges.
Let $\ol\CE$ denote the image of $\CE$ in the set $EX$ of non-oriented edges.
Further let
$$
\al=\frac12\sum_{\substack{x\to x'\notin\CE\\ x'\to x\in\CE}}w(x,x')^2.
$$
\end{definition}

\begin{theorem}
The entire function $Z_\CE(u)^{-1}$ can be written as
\begin{align*}
Z_\CE(u)^{-1}&=\det\(B_\CE(-u)\)\,\prod_{e\notin\ol\CE} \(1-u^2W(e)\)\, e^{-\al u^2},
\end{align*}
where the first factor can also  be expressed as
$$
\det\(B_\CE(-u)\)=\det\(1-uB_{\CE,1}+u^2B_{\CE,s}-u^3\tau J^2(1+uJ)^{-1}\sigma\).
$$

\end{theorem}

\begin{proof}
The trace of $\frac1u[B_{\CE,\ev}(u)-1]$ equals the logarithmic derivative of the infinite product minus $2\al u$.
From this we get the claim up to a constant nonzero factor.
Setting $u=0$ shows that the factor is one.
\end{proof}

{\small Mathematisches Institut\\
Auf der Morgenstelle 10\\
72076 T\"ubingen\\
Germany\\
\tt deitmar@uni-tuebingen.de}

\today


\begin{bibdiv} \begin{biblist}

\bib{Angel}{article}{
   author={Angel, Omer},
   author={Friedman, Joel},
   author={Hoory, Shlomo},
   title={The non-backtracking spectrum of the universal cover of a graph},
   journal={Transactions AMS},
   date={2014},
}

\bib{Bass}{article}{
   author={Bass, Hyman},
   title={The Ihara-Selberg zeta function of a tree lattice},
   journal={Internat. J. Math.},
   volume={3},
   date={1992},
   number={6},
   pages={717--797},
   issn={0129-167X},
}

\bib{BassLub}{book}{
   author={Bass, Hyman},
   author={Lubotzky, Alexander},
   title={Tree lattices},
   series={Progress in Mathematics},
   volume={176},
   note={With appendices by Bass, L. Carbone, Lubotzky, G. Rosenberg and J.
   Tits},
   publisher={Birkh\"auser Boston Inc.},
   place={Boston, MA},
   date={2001},
   pages={xiv+233},
   isbn={0-8176-4120-3},
}

\bib{Chinta}{article}{
   author={Chinta, G.},
   author={Jorgenson, J.},
   author={Karlsson, A.},
   title={Heat kernels on regular graphs and generalized Ihara zeta functions},
   journal={Monatshefte für Mathematik},
   date={2014},
}

\bib{Clair1}{article}{
   author={Clair, Bryan},
   author={Mokhtari-Sharghi, Shahriar},
   title={Zeta functions of discrete groups acting on trees},
   journal={J. Algebra},
   volume={237},
   date={2001},
   number={2},
   pages={591--620},
   issn={0021-8693},
}

\bib{Clair2}{article}{
   author={Clair, Bryan},
   title={Zeta functions of graphs with $\mathbb Z$ actions},
   journal={J. Combin. Theory Ser. B},
   volume={99},
   date={2009},
   number={1},
   pages={48--61},
   issn={0095-8956},
}

\bib{Grig}{article}{
   author={Grigorchuk, Rostislav I.},
   author={{\.Z}uk, Andrzej},
   title={The Ihara zeta function of infinite graphs, the KNS spectral
   measure and integrable maps},
   conference={
      title={Random walks and geometry},
   },
   book={
      publisher={Walter de Gruyter GmbH \& Co. KG, Berlin},
   },
   date={2004},
   pages={141--180},
}

\bib{Guido1}{article}{
   author={Guido, Daniele},
   author={Isola, Tommaso},
   author={Lapidus, Michel L.},
   title={Ihara zeta functions for periodic simple graphs},
   conference={
      title={C*-algebras and elliptic theory II},
   },
   book={
      series={Trends Math.},
      publisher={Birkh\"auser},
      place={Basel},
   },
   date={2008},
   pages={103--121},
}

\bib{Guido2}{article}{
   author={Guido, Daniele},
   author={Isola, Tommaso},
   author={Lapidus, Michel L.},
   title={Ihara's zeta function for periodic graphs and its approximation in
   the amenable case},
   journal={J. Funct. Anal.},
   volume={255},
   date={2008},
   number={6},
   pages={1339--1361},
   issn={0022-1236},
}

\bib{Hash0}{article}{
   author={Hashimoto, Ki-ichiro},
   author={Hori, Akira},
   title={Selberg-Ihara's zeta function for $p$-adic discrete groups},
   conference={
      title={Automorphic forms and geometry of arithmetic varieties},
   },
   book={
      series={Adv. Stud. Pure Math.},
      volume={15},
      publisher={Academic Press},
      place={Boston, MA},
   },
   date={1989},
   pages={171--210},
}

\bib{Hash1}{article}{
   author={Hashimoto, Ki-ichiro},
   title={Zeta functions of finite graphs and representations of $p$-adic
   groups},
   conference={
      title={Automorphic forms and geometry of arithmetic varieties},
   },
   book={
      series={Adv. Stud. Pure Math.},
      volume={15},
      publisher={Academic Press},
      place={Boston, MA},
   },
   date={1989},
   pages={211--280},
}

\bib{Hash2}{article}{
   author={Hashimoto, Ki-ichiro},
   title={On zeta and $L$-functions of finite graphs},
   journal={Internat. J. Math.},
   volume={1},
   date={1990},
   number={4},
   pages={381--396},
   issn={0129-167X},
}

\bib{Hash3}{article}{
   author={Hashimoto, Ki-ichiro},
   title={Artin type $L$-functions and the density theorem for prime cycles
   on finite graphs},
   journal={Internat. J. Math.},
   volume={3},
   date={1992},
   number={6},
   pages={809--826},
   issn={0129-167X},
}
	

\bib{Hash4}{article}{
   author={Hashimoto, Ki-ichiro},
   title={Artin $L$-functions of finite graphs and their applications},
   language={Japanese},
   note={Algebraic combinatorics (Japanese) (Kyoto, 1992)},
   journal={S\=uri\-kaise\-kikenky\=usho K\=oky\=uroku},
   number={840},
   date={1993},
   pages={70--81},
}

\bib{Ihara1}{article}{
   author={Ihara, Yasutaka},
   title={On discrete subgroups of the two by two projective linear group
   over ${\germ p}$-adic fields},
   journal={J. Math. Soc. Japan},
   volume={18},
   date={1966},
   pages={219--235},
   issn={0025-5645},
}

\bib{Ihara2}{article}{
   author={Ihara, Yasutaka},
   title={Discrete subgroups of ${\rm PL}(2,\,k_{\wp })$},
   conference={
      title={Algebraic Groups and Discontinuous Subgroups (Proc. Sympos.
      Pure Math., Boulder, Colo., 1965)},
   },
   book={
      publisher={Amer. Math. Soc.},
      place={Providence, R.I.},
   },
   date={1966},
   pages={272--278},
}


\bib{Sunada}{article}{
   author={Kotani, Motoko},
   author={Sunada, Toshikazu},
   title={Zeta functions of finite graphs},
   journal={J. Math. Sci. Univ. Tokyo},
   volume={7},
   date={2000},
   number={1},
   pages={7--25},
   issn={1340-5705},
}

\bib{Lub1}{article}{
   author={Lubotzky, Alexander},
   title={Trees and discrete subgroups of Lie groups over local fields},
   journal={Bull. Amer. Math. Soc. (N.S.)},
   volume={20},
   date={1989},
   number={1},
   pages={27--30},
   issn={0273-0979},
}

\bib{Lub2}{article}{
   author={Lubotzky, Alexander},
   title={Lattices in rank one Lie groups over local fields},
   journal={Geom. Funct. Anal.},
   volume={1},
   date={1991},
   number={4},
   pages={406--431},
   issn={1016-443X},
}

\bib{Scheja}{article}{
   author={Scheja, Ortwin},
   title={On zeta functions of arithmetically defined graphs},
   journal={Finite Fields Appl.},
   volume={5},
   date={1999},
   number={3},
   pages={314--343},
   issn={1071-5797},
}

\bib{Serre}{book}{
   author={Serre, Jean-Pierre},
   title={Trees},
   series={Springer Monographs in Mathematics},
   note={Translated from the French original by John Stillwell;
   Corrected 2nd printing of the 1980 English translation},
   publisher={Springer-Verlag},
   place={Berlin},
   date={2003},
   pages={x+142},
   isbn={3-540-44237-5},
}

\bib{Simon}{book}{
   author={Simon, Barry},
   title={Trace ideals and their applications},
   series={Mathematical Surveys and Monographs},
   volume={120},
   edition={2},
   publisher={American Mathematical Society},
   place={Providence, RI},
   date={2005},
   pages={viii+150},
   isbn={0-8218-3581-5},
}

\bib{ST1}{article}{
   author={Stark, H. M.},
   author={Terras, A. A.},
   title={Zeta functions of finite graphs and coverings},
   journal={Adv. Math.},
   volume={121},
   date={1996},
   number={1},
   pages={124--165},
   issn={0001-8708},
}

\bib{ST2}{article}{
   author={Stark, H. M.},
   author={Terras, A. A.},
   title={Zeta functions of finite graphs and coverings. II},
   journal={Adv. Math.},
   volume={154},
   date={2000},
   number={1},
   pages={132--195},
   issn={0001-8708},
}


\bib{ST3}{article}{
   author={Terras, A. A.},
   author={Stark, H. M.},
   title={Zeta functions of finite graphs and coverings. III},
   journal={Adv. Math.},
   volume={208},
   date={2007},
   number={1},
   pages={467--489},
   issn={0001-8708},
}

\bib{Sun1}{article}{
   author={Sunada, Toshikazu},
   title={$L$-functions in geometry and some applications},
   conference={
      title={Curvature and topology of Riemannian manifolds},
      address={Katata},
      date={1985},
   },
   book={
      series={Lecture Notes in Math.},
      volume={1201},
      publisher={Springer},
      place={Berlin},
   },
   date={1986},
   pages={266--284},
}

\bib{Sun2}{article}{
   author={Sunada, Toshikazu},
   title={Fundamental groups and Laplacians},
   conference={
      title={Geometry and analysis on manifolds},
      address={Katata/Kyoto},
      date={1987},
   },
   book={
      series={Lecture Notes in Math.},
      volume={1339},
      publisher={Springer},
      place={Berlin},
   },
   date={1988},
   pages={248--277},
}

\end{biblist} \end{bibdiv}
\end{document}